\documentclass[11pt]{amsart}

\newcommand{\al}{\alpha}

\newcommand{\dd}{\delta}
\newcommand{\ee}{\varepsilon}

\newcommand{\C}{\mathcal{C}}
\newcommand{\N}{\mathcal{N}}
\newcommand{\B}{\mathcal{B}}

\newcommand{\R}{\mathbb{R}}
\newcommand{\Z}{\mathbb{Z}}

\newcommand{\len}{\mathrm{Length}}
\newcommand{\area}{\mathrm{Area}}
\newcommand{\vol}{\mathrm{Vol}}

\newcommand{\D}{\mathcal D}

\newcommand{\vep}{\varepsilon}

\newcommand{\eps}{\varepsilon}
\newcommand{\diam}{\mathrm{diam}\,}

\usepackage{mathtools}
\usepackage{plankssty}
\usepackage{graphicx}
\usepackage{todonotes}

\usepackage{comment}

\newtheoremstyle{note}%                % Name
  {}%                                     % Space above
  {}%                                     % Space below
  {\itshape}%                             % Body font
  {}%                                     % Indent amount
  {\bfseries}%                            % Theorem head font
  {}%                                    % Punctuation after theorem head
  {}%                                    % Space after theorem head, ' ', or \newline
  {}%                                     % Theorem head spec (can be left empty, meaning `normal')
\makeatletter
\def\th@plain{%
  \thm@notefont{}% same as heading font
  \itshape % body font
}
\def\th@note{%
  \thm@notefont{}% same asheading font
  \itshape % body font
}
\makeatother

\theoremstyle{plain} % other options: definition, remark
\newtheorem{theorem}{Theorem}[section]
\newtheorem{conjecture}[theorem]{Conjecture}
\newtheorem{lemma}[theorem]{Lemma}
\newtheorem{claim}[theorem]{Claim}

\theoremstyle{definition}
\theoremstyle{remark}
\newtheorem{remark}{Remark}

\theoremstyle{note}

\usepackage[margin=1.44in]{geometry}

\usepackage{fancyhdr}
\title{On plank coverings of cube annuli}
\author{Gergely Ambrus
\and Matthew Corbo
\and Dakota Nobuyuki Goto
\and Desmond Leitz
\and Kai Mawhinney}

\begin{document}

\date{\today}

\begin{abstract}

We study plank coverings of cube annuli $Q\setminus\delta Q\subset\mathbb{R}^d$, where $Q$ is the $d$-dimensional unit cube and $\delta Q$ is a homothetic copy lying in its interior. A.~Bezdek proved that, in dimension $2$, every plank covering of such an annulus has total width at least~1 whenever $\delta\le1-1/\sqrt2$. Smurov, Bogataya, and Bogatyi extended this result to all $d\ge2$ under the condition $\delta\le1-1/\sqrt d$. It is conjectured that the same lower bound holds for $\delta\le1-1/d$. We provide evidence for this conjecture in dimensions $2$ and $3$ by proving that there exist positive constants $\varepsilon_2$ and $\varepsilon_3$ such that the bound remains valid whenever $\delta\le1-1/\sqrt d+\varepsilon_d$.
\end{abstract}

\maketitle 

\vspace{-0.3in}

\section{Introduction}
In $\R^d$, a {\em plank} is the closed space between two parallel affine hyperplanes, whose distance apart is the \textit{width} of the plank. We are interested in the following question: given a subset $S\subset \mathbb{R}^d$, how can one cover $S$ with planks such that the sum of the plank widths is minimized? Work on this question dates back nearly a century. In 1932, Tarski \cite{tarski1932} conjectured that any plank covering of a convex body $K \subset \R^d$ must have total plank width at least $\omega(K)$, the minimum width of a single plank covering $K$. In other words, Tarski conjectured that the narrowest single plank that covers $K$ is optimal. In 1951, Bang~\cite{Bang1951} proved the conjecture.

There are many varieties of plank problems. We direct the interested reader to Verreault's recent survey paper for more details \cite{verreault2026plank}. In particular, non-convex shapes are the source of many open questions. In this paper, we study plank coverings of \textit{cube annuli} and improve results of A. Bezdek and Smurov, Bogataya, and Bogatyi.
A cube annulus $Q\setminus \delta Q$ is obtained from the unit cube $Q = [0,1]^d$ by the erasure of an open homothetic copy $\delta Q$ with scaling factor $0 < \delta <1$ within its interior. A. Bezdek \cite{Bezdek2003} proved that in the plane, it is impossible to cover such a square annulus by a collection of closed planks whose total width is less than 1 whenever $\delta \leq 1 - 1 /{\sqrt{2}}$.  Smurov, Bogataya, and Bogatyi~\cite{Smurov2010} generalized this bound to higher dimensions. These results can be combined into the following theorem.
\begin{theorem}[\cite{Bezdek2003,Smurov2010}]\label{thm:bezdek-smurov}
  Let $d \geq 2$ be an integer. If $\delta \leq 1 - 1/\sqrt d$, then every plank covering of $Q \setminus \delta Q$ has total width at least 1.
\end{theorem}
It is natural to believe that the optimal value of $\delta$ is manifested by the straightforward coordinate plank covering, described in Subsection~\ref{subseq:cube}. This leads to the following conjecture: 

\begin{conjecture}[\cite{Bezdek2003,Smurov2010}]\label{conj:bezdek}
    If $\delta \leq 1 - 1/d$, every plank covering of $Q \setminus \delta Q$ has total width at least~1.
\end{conjecture}

Our work improves \Cref{thm:bezdek-smurov} in dimensions 2 and 3 by extending the range of $\dd$ and hence showing that in these cases, the bound $\dd\le 1-1/\sqrt d$ is not sharp. This result provides a step towards proving \Cref{conj:bezdek} for $d=2,3$.

\begin{theorem}\label{thm:main}
For $d=2, 3$, there exist $\eps_d>0$ such that the following holds. If
    \[\delta_d\le 1-1/\sqrt{d}+\eps_d,\]
then any finite collection of closed planks that cover the annulus $Q_d\setminus \delta_d Q_d\subset \R^d$ has total width at least~1.
\end{theorem}

As shown by Remark~\ref{rk:open planks} below, it suffices to prove Theorem~\ref{thm:main} for collections of open planks.

We mention some further results related to plank coverings of annuli obtained from convex bodies. White and Wisewell~\cite{WhiteWisewell2007} characterized the convex polygons for which removing any sufficiently small homothetic copy leaves the minimum total covering width unchanged. Ambrus, Huddell, Lai, Quirk, and Williams~\cite{AmbrusEtAl2025} extended their counterexample construction to planar and spatial convex bodies that are spiky in a minimal width direction: suitably placed homothetic holes of arbitrarily small size allow coverings with total width strictly below the minimal width of the original body. Very recently, Bakaev and Yehudayoff~\cite{BakaevYehudayoff2026} proved a related conjecture of A. Bezdek by showing that every finite plank covering of a circular disk with a sufficiently small concentric disk removed can be rearranged to cover the entire disk.

After this brief historical overview, we continue with some definitions and notational conventions.

\subsection{General notation}

Throughout the paper, we work in Euclidean space $\mathbb{R}^d$ with $d \geq 2$. Let $Q_d \coloneqq [0,1]^d \subset \mathbb{R}^d$ be the unit cube in $\R^d$. We will use the notation $Q = Q_d$, and the dimension $d$ will always be clear from the context.

Homotheties are defined as follows. For a scaling factor $\delta > 0$ and a homothetic center $p\in\mathbb{R}^d$, write $\delta_p S$ for the image of $S \subset \R^d$ under the map $x \mapsto p + \delta(x - p)$. We say that $\delta_p S$ is the \textit{homothetic image of $S$ about $p$ by $\delta$}. Let $(\cdot)^\circ$ and $\overline{(\cdot )}$  denote topological interior and topological closure. We are interested in the \textit{cube annulus} $Q \setminus (\delta_p Q)^\circ$, where $\delta\in (0,1)$ and $p\in Q^\circ$. For brevity, we write
    \[
    Q \setminus \delta Q \ \coloneqq\  Q \setminus (\delta_p Q)^\circ
    \]
for a cube annulus. We have suppressed the homothetic center $p$ in our notation because our arguments work whenever $\dd_p Q\subset Q$. 

We also introduce \emph{centered} annuli. Given a thickness parameter $a\in (0,1/2)$, let $Q_a = (a,1-a)^d\subset \R^d$. We define
\begin{equation}\label{eq:Adef}
A_a = Q\setminus Q_a
\end{equation}
to be the closed centered annulus of thickness $a$. The dimension $d$ will be clear from context. 
For $d=1, 2, 3$, the Lebesgue measure on $\R^d$ will be called length, area, and volume, and denoted by  $\len(\cdot )$, $\area(\cdot )$, and $\vol(\cdot )$, respectively. We write $\vol_d(\cdot)$ for $d\ge 4$. For $p \in [1, \infty]$, we will denote the $\ell_p$-norm on $\R^d$ by $\|\cdot \|_p$. The Euclidean norm $\|\cdot\|_2$, or length, will simply be denoted by $\|\cdot\|$. 

A vector $v \in \R^d$ is called {\em diagonal} if its coordinates are identical up to sign. Among vectors with a fixed length, such vectors maximize the $\ell_1$-norm: 
\begin{claim}\label{claim:max1norm}
    For any $v\in\mathbb{R}^d$,
    \begin{equation*}
\|v\|_1\le \sqrt d\|v\|_2
\end{equation*}
with equality if and only if $v$ is diagonal.
\end{claim}
\begin{proof}
For this argument exclusively, let $v_1, \ldots, v_d$ denote the coordinates of $v$. Then
\begin{align*}
\|v\|_1 = \sum_{i=1}^d |v_i| \leq \sqrt{d} \cdot \sqrt{\sum_{i=1}^d v_i^2} = \sqrt{d}   \cdot \|v\|_2 
\end{align*}
by the Cauchy-Schwarz inequality. Equality holds if all $|v_i|$ are identical.
\end{proof}

Given two lines, or a vector and a line in the plane, their {\em angle} is understood to be the smaller nonnegative angle $\theta\in[0,\frac\pi2]$ enclosed by them. The angle of a line is its  angle with the $x$-axis; for a planar plank, its angle is that of its central line.

For $v\in\R^2$, $\|v\|_1$ depends only on $\|v\|$ and the angle
between $v$ and the $x$-axis: writing
$v=\|v\|(\pm\cos\theta,\pm\sin\theta)$ gives
\[\|v\|_1 = \|v\|\left(\cos\theta+\sin\theta\right).\]
In particular, rotating $v$ through $\frac\pi2$ replaces $\theta$ by
$\frac\pi2-\theta$ and hence leaves $\|v\|_1$ unchanged.

Lastly, given a vector $v \in \R^d$, we will denote by $\pi_v$ the orthogonal projection onto the linear span of $v$. Working in the standard basis $e_1, \ldots, e_d$ of $\R^d$, we will simply denote $\pi_{e_i}$ by $\pi_i$ for $i \in [d]:= \{ 1, \ldots, d\}$. We say the \emph{width} of a subset $S\subset \R^d$ in the direction of standard basis vector $e_i$ is $\mathrm{diam}(\pi_i(S))$, the diameter of the set $\pi_i(S)$.

\subsection{Planks}\label{subsection:planks}

A \emph{plank} in $\R^d$ is the closed region
between two distinct parallel affine hyperplanes. It can be written as 
    \begin{equation*}\label{eq:plankdef}
    P(v,t)\ \coloneqq\ \Big\{x\in\R^d:\ \Big| \Big\langle x,\frac v {\| v\|} \Big \rangle - t \Big| \le  \frac {\|v\|}{2} \Big \}
    \end{equation*}
where $v \in \R^d$ is nonzero. We use the conventions that
\begin{itemize}
    \item $v = v(P)\in \mathbb{R}^d$ is the \textit{cross-normal vector} of $P$, see \Cref{fig:plank-in-r2};
    \item $u = u(P) := v/\norm v$ is the \textit{unit normal vector} of $P$;
    \item $w(P) \coloneqq \norm v$ is the \textit{width} of $P$;
    \item $t\in \R$ is the signed distance from the \emph{central hyperplane}
$\{x\in\R^d:\langle x,u\rangle = t\}$ to the origin; when $d=2$ we call this
line the \emph{midline} of $P$.
\end{itemize}
An \emph{open} plank is the interior of a closed plank, with the above quantities being defined for it as those of its closed counterpart.

\begin{remark}\label{rk:open planks}
    Suppose there exists a covering of some set $S\subset \R^d$ by a collection of closed planks $\C$ with $w(\C)<1$. One may slightly expand these planks so as to maintain the total width being less than 1, while obtaining a collection of {\em open} planks that still covers $S$. This implies there also exists a collection of open planks with total width less than 1 that covers $S$.
\end{remark}

\begin{figure}[H]
\centering
\begin{tikzpicture}[scale=1.2]

% Draw the plank in a rotated coordinate system
\begin{scope}[rotate=14]

  % Plank fill
  \fill[blue!50!green, opacity=0.25] (-2.8,-0.35) rectangle (2.8,0.35);

  % Boundary lines
  \draw[thick] (-2.8,-0.35) -- (2.8,-0.35);
  \draw[thick] (-2.8,0.35) -- (2.8,0.35);

  % Central line
  \draw[dashed] (-2.8,0) -- (2.8,0);

  % cross-normal vector
  \draw[thick,->] (0,-0.35) -- (0,0.35);
  \node[right] at (.02,.21) {$v$};

  % Plank label
  \node at (1.8,0.55) {$P$};

\end{scope}

\end{tikzpicture}
\caption{A plank in $\mathbb{R}^2$ and its cross-normal vector.}
\label{fig:plank-in-r2}
\end{figure}
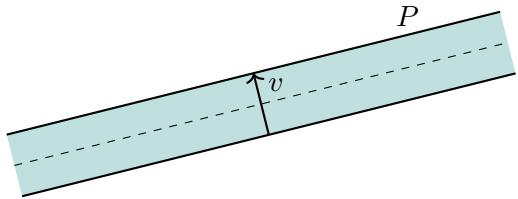

Throughout the proof, a {\em collection of planks} will always mean a finite set of open planks. Some constructions use closed planks though. 

Given a collection $\C$ of planks, we define the \textit{total width} $w(\C)$ to be
\begin{equation*}\label{eq:totalwidth}
w(\C) = \sum_{P \in \C} w(P).
\end{equation*}
Planks in $\C$ will be denoted $P_1, \ldots, P_n$. For $i\in \{1,\dots , n\}$, let $v_i$ and $u_i$ denote the cross-normal vector and unit normal vector of $P_i$ respectively.  

When talking about covering, we identify $\C$ by the union of its members. Thus, we say that $\C$ \textit{covers} a set $X \subset \R^d$ if $X \subset \bigcup_{P\in \C} P$. We say $X$ is \textit{uncovered} by $\C$ otherwise. A \textit{plank covering} of a set $S$ is a covering of $S$ by a finite collection of open planks. 

A key tool in the theory of plank coverings is Bang's Lemma.

\begin{lemma}[\cite{Bang1951}]\label{lem:bang}
    Let $\C$ denote a collection of open planks $P_1, \ldots, P_n \subset \R^d$ with cross-normal vectors $v_1, \ldots, v_n$. Then any translate of the set
    \[X_\C \coloneqq \left\{\sum_{i=1}^n\frac12\vep_i v_i\mid\vep_i=\pm 1\right\}\]
    is uncovered by $\C$. 
\end{lemma}

We note that the diameter of the set $X_{\C}$ is at most $\sum |v_i| = w(\C)$.

\subsection{Plank coverings of cube annuli} \label{subseq:cube}

Heuristically, the `best' ways to cover $Q\setminus \delta Q$ are the two most natural ones, which are depicted in \Cref{fig:natural-coverings}. They are:
\begin{enumerate}
    \item[(A)] Cover $Q\setminus \delta Q$ with a single plank of width 1;
    \item[(B)] Put planks on all sides of the hole $\dd Q$.
\end{enumerate}

\begin{figure}[h!]
    \centering
    \begin{tikzpicture}[scale=3.8]
    
      \def\halfside{0.25}   
      \def\cx{0.10}         
      \def\cy{-0.05}        
    
      \def\xL{\cx-\halfside}  
      \def\xR{\cx+\halfside}  
      \def\yB{\cy-\halfside}  
      \def\yT{\cy+\halfside}  
    
      \begin{scope}[xshift=1.9cm]
        \coordinate (A) at (-0.5,-0.5);
        \coordinate (B) at ( 0.5,-0.5);
        \coordinate (C) at ( 0.5, 0.5);
        \coordinate (D) at (-0.5, 0.5);
    
        \coordinate (Ai) at (\xL,\yB);
        \coordinate (Bi) at (\xR,\yB);
        \coordinate (Ci) at (\xR,\yT);
        \coordinate (Di) at (\xL,\yT);
    
        \fill[gray!22, even odd rule]
          (A)--(B)--(C)--(D)--cycle
          (Ai)--(Bi)--(Ci)--(Di)--cycle;
    
        \begin{scope}
          % \clip[even odd rule];
            %(A)--(B)--(C)--(D)--cycle;
            % (Ai)--(Bi)--(Ci)--(Di)--cycle;
    
          \fill[blue!40,opacity=.30]
            (-0.5,-0.7) rectangle (\xL,0.7);
          \draw[dashed,blue!70] (-0.5,-0.7) -- (-0.5,0.7);
          \draw[dashed,blue!70] (\xL,-0.7) -- (\xL,0.7);
    
          \fill[blue!40,opacity=.30]
            (\xR,-0.7) rectangle (0.5,0.7);
          \draw[dashed,blue!70] (\xR,-0.7) -- (\xR,0.7);
          \draw[dashed,blue!70] (0.5,-0.7) -- (0.5,0.7);
    
          \fill[green!40,opacity=.25]
            (-0.7,-0.5) rectangle (0.7,\yB);
          \draw[dashed,green!70] (-0.7,-0.5) -- (0.7,-0.5);
          \draw[dashed,green!70] (-0.7,\yB) -- (0.7,\yB);
    
          \fill[green!40,opacity=.25]
            (-0.7,\yT) rectangle (0.7,0.5);
          \draw[dashed,green!70] (-0.7,\yT) -- (0.7,\yT);
          \draw[dashed,green!70] (-0.7,0.5) -- (0.7,0.5);
        \end{scope}
    
        % Outer and inner squares
        \draw[thick] (A)--(B)--(C)--(D)--cycle;
        \draw[thick] (Ai)--(Bi)--(Ci)--(Di)--cycle;
    
        % Labels
        \node[anchor=south east] at (D) {$Q$};
        \node[anchor=north west] at (Di) {$\frac12 Q$};
        \node at (0,-0.9) {(B) $2d$-plank covering};
      \end{scope}

      \begin{scope}[xshift=0cm]
        
        \coordinate (A2) at (-0.5,-0.5);
        \coordinate (B2) at ( 0.5,-0.5);
        \coordinate (C2) at ( 0.5, 0.5);
        \coordinate (D2) at (-0.5, 0.5);
    
        \coordinate (Ai2) at (\xL,\yB);
        \coordinate (Bi2) at (\xR,\yB);
        \coordinate (Ci2) at (\xR,\yT);
        \coordinate (Di2) at (\xL,\yT);
    
        \fill[gray!22, even odd rule]
          (A2)--(B2)--(C2)--(D2)--cycle
          (Ai2)--(Bi2)--(Ci2)--(Di2)--cycle;
    
        \begin{scope}
          % \clip[even odd rule]
          %   (A2)--(B2)--(C2)--(D2)--cycle
          %   (Ai2)--(Bi2)--(Ci2)--(Di2)--cycle;
    
          \def\wh{1.0} 
          \fill[orange!40,opacity=.30]
            (-0.7,-\wh/2) rectangle (0.7,\wh/2);
          \draw[dashed,orange!80] (-0.7,-\wh/2) -- (0.7,-\wh/2);
          \draw[dashed,orange!80] (-0.7,\wh/2) -- (0.7,\wh/2);
        \end{scope}
    
        \draw[thick] (A2)--(B2)--(C2)--(D2)--cycle;
        \draw[thick] (Ai2)--(Bi2)--(Ci2)--(Di2)--cycle;
    
        \node[anchor=south east] at (D2) {$Q$};
        \node[anchor=north west] at (Di2) {$\frac12 Q$};
        \node at (0,-0.7) {(A) One-plank covering};
      \end{scope}
    
    \end{tikzpicture}
    \caption{Natural ways to cover $Q\setminus \dd Q$, here $\dd = 1/2$.}
    \label{fig:natural-coverings}
\end{figure}
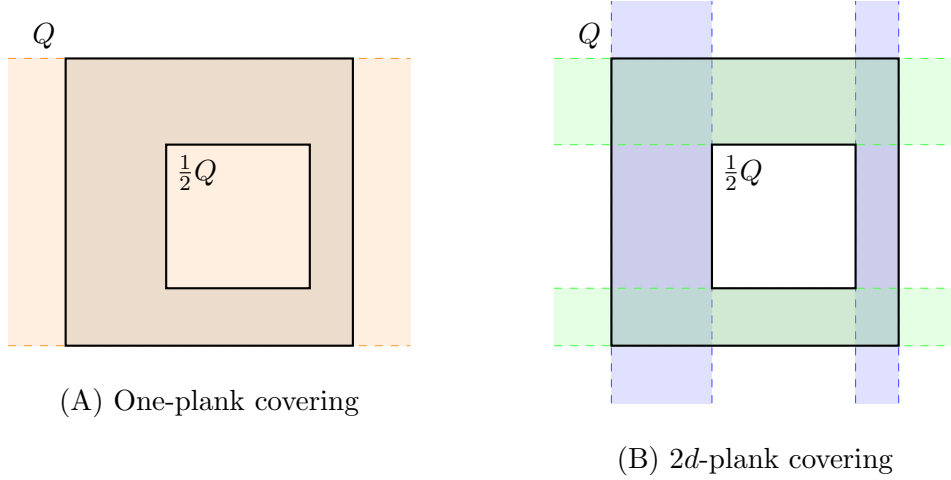

In $\R^2$, the covering (B) has smaller width than  (A) if and only if $\delta >\frac12$. This led A. Bezdek to conjecture that in the planar case,  total plank width at least 1 is required to cover $Q\setminus \dd Q$ for every $\dd\le \frac12$.

Applying an analogous technique in $\R^d$, we find (B) has smaller total width than (A) if and only if $\delta > 1-\frac1d$. As a corollary of this simple construction, we find that for $\dd>1-\frac1d$, total plank width at least 1 is not required to cover $Q\setminus \dd Q$. This observation is well known, and it motivates Conjecture~\ref{conj:bezdek}.

\subsection{Bezdek's approach}

We describe the proof of \Cref{thm:bezdek-smurov} because it introduces key techniques that are used to prove  \Cref{thm:main}. The crucial observation is the following fact.

\begin{lemma}\label{lemma:new plank} Let $d\ge 2$ be an integer. Assume that a collection $\C$ of open planks covers $Q\setminus \delta Q\subset \R^d$ for some $\delta \in (0,1)$. Then $X_\C$ cannot be translated into any $d$-cell of the form
    \begin{equation*}\label{eq:Rj}
        R_j = [0,1]\times\cdots\times \underbrace{[0,1-\delta]}_{j\text{th coordinate}}\times\cdots\times [0,1]
    \end{equation*}
for any $j\in [d]$.
\end{lemma}

\begin{proof}
    Suppose for contradiction that we can translate $X_\C$ into some $R_j$. Then the width of $X_\C$ in direction $e_j$ is at most $1-\delta$, while its width in any other coordinate direction is at most 1. 
    Let $P$ be an open plank with cross-normal vector $\dd e_j$ such that $P$ covers $(\delta Q_d)^\circ$. Then the collection $\C' := \C \cup \{P\}$ covers $Q$. Since $P$ has width $\delta$, the Bang point set $X_{\C'}$ of the extended collection is obtained as
    \[
    X_{\C'} = \left(X_\C + \frac \delta 2 e_j\right) \cup  \left(X_\C - \frac \delta 2 e_j\right).
    \]
By the above observations, the width of $X_{\C'}$ in any coordinate direction is at most~$1$.

Therefore, $X_{\C'}$ can be translated into $Q$. Because $\C'$ covers $Q$, $\C'$ must also cover a translate of $X_{\C'}$. This contradicts \Cref{lem:bang}.
\end{proof}

\begin{proof}[Proof of Theorem~\ref{thm:bezdek-smurov}]
Suppose for contradiction that there exists a collection $\C$ which covers $Q \setminus \delta Q$ and has $w(\C)< 1$. The width of the corresponding Bang point set in direction $e_j$ is given by 
   \[
   \diam(\pi_j(X_\C)) = \sum_{P \in \C}  \|\pi_j (v(P)) \|.
   \]
Since
\[
\sum_{P\in\mathcal C}\|\pi_j(v(P))\|
\le\sum_{P\in\mathcal C}\|v(P)\|
=w(\mathcal C)<1,
\]
we obtain that  $\diam(\pi_j(X_\C))$ is less than 1 for any $j \in [d]$.  By \Cref{lemma:new plank}, $X_\mathcal{C}$ cannot be translated into any $d$-cell $R_1,\dots,R_d$. In the present setting, this implies that the width of $X_\C$ in each direction $i\in [d]$ exceeds $1-\dd$:
 \[
    \sum_{P \in \C} \| \pi_i (v(P)) \| > 1 - \delta.
   \]
Summing over $i \in [d]$ we obtain 
 \begin{equation*}\label{eq:vp1}
    \sum_{P \in \C} \| v(P) \|_1 > d (1 - \delta).
   \end{equation*}
By \Cref{claim:max1norm}, 
\begin{equation*}\label{eq:vp2}
     w(\C) = \sum_{P \in \C} \| v(P) \|_2 \ge \frac{1}{\sqrt d} \sum_{P \in \C} \| v(P) \|_1> \sqrt{d} (1 - \delta).
\end{equation*}   
By our initial assumption, $w(\C) <1$. Therefore, $\delta > 1 - 1/\sqrt{d}$, proving the theorem.
\end{proof}

Since equality in \Cref{claim:max1norm} can only hold if $v$ is \emph{diagonal}, the above inequalities are tight only if all planks are diagonal. On the other hand, we will show that nearly-diagonal planks are ineffective at covering the boundary of a cube in dimensions $2$ and $3$. This is the idea behind the proof of Theorem~\ref{thm:main}.

\section{An Improved Bound in the Planar Case}
\label{sec:improved-bound-2d}
This section is dedicated to proving \Cref{thm:main} for $d=2$. Throughout the section, we work in the plane. Accordingly, $Q$ denotes the unit square, and $A_a$ stands for the centered square annulus $Q\setminus Q_a$ of thickness $a\in (0,1/2)$, as defined in \eqref{eq:Adef}. We will refer to the sets which are within $a$ of a given side of $Q$ (explicitly, $[0,1]\times [0,a],[0,1]\times [1-a,1],[0,a]\times [0,1],[1-a,1]\times [0,1]$) as the \emph{sides} of $A_a$ (note that these overlap; see \Cref{fig:annulus-l1-l2}).

We prove the statement by the following strategy. Let $\delta = 1 - 1/\sqrt{2} + \ee$ for some small $\ee >0$, and suppose for contradiction that $Q \setminus \delta Q$ can be covered by a collection $\C$ of planks of total width less than 1. \Cref{lemma:new plank} implies that most of the total plank width is carried by nearly diagonal planks. For a small positive $a>0$, consider $A_a$, which is a thin collar around the boundary of $Q$ with area $4 a - 4 a^2$. This must be covered by the union of the planks in $\C$ and the cut-out square $\delta Q$. However, using that $\C$ essentially consists of diagonal planks, its members can cover only roughly $2 \sqrt 2 a$ area  of $A_a$, while the square $\delta Q$ covers at most $2 \delta a = (2 - \sqrt{2}) a + 2 \ee a$ area of it. Altogether, at leading order, this would force
\[
4a \leq (2 + \sqrt{2}) a,
\]
a contradiction. The difficulty in making this argument precise is that even a small number of outlier planks in $\C$ that are parallel to a side of $Q$ can cover a relatively large area of $A_a$.

We will bound the area near the boundary of $Q$ that may be covered by `nearly-diagonal' planks. To that end, we first estimate the length of the intersection of a line with the collar~$A_a$.

\begin{figure}
\centering
\begin{tikzpicture}[scale=2.5]

  \coordinate (A)  at (-1,-1);
  \coordinate (B)  at ( 1,-1);
  \coordinate (C)  at ( 1, 1);
  \coordinate (D)  at (-1, 1);

  \coordinate (Ai) at (-0.4,-0.4);
  \coordinate (Bi) at ( 0.4,-0.4);
  \coordinate (Ci) at ( 0.4, 0.4);
  \coordinate (Di) at (-0.4, 0.4);

  \fill[gray!42, even odd rule]
    (A)--(B)--(C)--(D)--cycle
    (Ai)--(Bi)--(Ci)--(Di)--cycle;

  % Highlight the left side of A
  \fill[purple!70, opacity=0.5] (-1,-1) rectangle (-0.4,1);
  % Highlight the top side of A
  \fill[yellow!70, opacity=0.5] (-1,1) rectangle (1,0.4);

  \draw[thick] (A)--(B)--(C)--(D)--cycle;
  \draw[thick] (Ai)--(Bi)--(Ci)--(Di)--cycle;

  \node[anchor=south east] at (D)  {$Q$};
  \draw[<->] (-.4,-.5) -- (.4,-.5);
  \node[below] at (0,-.5) {$(1-2a)$};

  \draw[gray!80, dashed] (-1.2,-1.0) -- (1,1.2) node[above right] {$\ell$};

  \draw[line width=1.1pt,blue]
    (-1.0,-0.8) -- (-0.4,-0.2);
  \node[blue,anchor=north east] at (-0.625,-0.25) {$s_1$};

  \draw[line width=1.1pt,red]
    (0.2,0.4) -- (0.8,1.0);
  \node[red,anchor=south west] at (0.35,0.7) {$s_2$};

  \draw[<->] (0,0.4) -- (0,1.0);
  \node[left] at (0,0.7) {$a$};
  \node at (0,-0.27) {$Q_a$};

\end{tikzpicture}
\caption{The annulus $A_a$ and the line $\ell$ intersecting it in two segments, $s_1$ and $s_2$. The left \emph{side} of $A_a$ is highlighted in purple, and the top \emph{side} of $A_a$ is highlighted in yellow (their intersection is orange).}
\label{fig:annulus-l1-l2}
\end{figure}
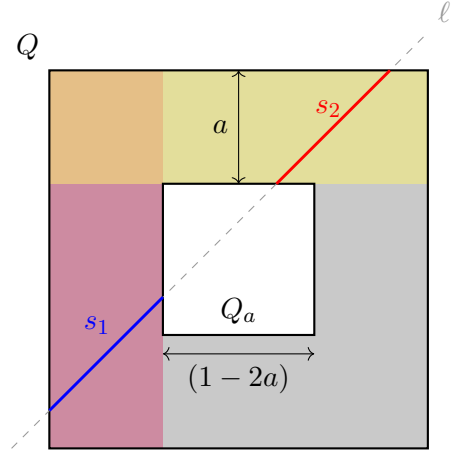

\begin{claim}\label{claim:2}
    Suppose $\ell$ is a line in $\R^2$ of angle $\theta\in (0,\pi/2)$. %    with the $x$-axis. 
    Then the length of intersection of $\ell$ with the annulus $A_a$ is no more than
    \[L := \len(\ell\cap A_a) \leq a\left( \frac{1}{\sin\theta}+ \frac{1}{\cos\theta}\right) =2a \, \frac{\cos\theta+\sin\theta}{\sin2\theta} \,.\]
\end{claim}

\begin{proof}
By symmetry, we may suppose $\theta\ge\frac{\pi}{4}$: reflection in the line $y=x$ preserves $A_a$ and replaces $\theta$ by $\frac{\pi}{2}-\theta$, interchanging $\sin\theta$ with $\cos\theta$, so neither $L$ nor the claimed bound changes. 

We deduce how $L$ depends on $\theta$. We claim that $\ell\cap A_a$ is always contained in the union of two sides of $A_a$. Indeed, if $\ell$ does not intersect $\overline{Q}_a$, then we may translate $\ell$ inwards until it contains a corner of $\overline{Q}_a$ without decreasing $L$. 
    
Thus, we may assume that $\ell$ intersects $\overline{Q}_a$. If $\ell$ is the extension of a side of $Q_a$, then $\ell \cap Q$ is contained in the corresponding side of $A_a$. If $\ell$ is not an extension of a side of $Q_a$, then by convexity, $\ell$ intersects the boundary of $\overline{Q}_a$ in at most two points, and by selecting two sides of $\overline{Q}_a$ that contain these points, the union of the two corresponding sides $S_1,S_2$ of $A_a$ contains $\ell\cap A_a$. For $i=1,2$, let $s_i$ be the length of the segment $\ell \cap S_i$ (note that these may overlap).
    
Since $\theta\ge\frac{\pi}{4}$, the only possibility for both $S_1$ and $S_2$ being vertical is when $\ell$ is precisely diagonal, in which case we may swap one of them for a horizontal side. Thus, we can always assume that one of $S_1$ and $S_2$ is horizontal. 

The length of intersection of $\ell$ with a horizontal side is at most $\frac{a}{\sin\theta}$, and with a vertical side it is at most $\frac{a}{\cos\theta}$. Because $\theta\ge \frac{\pi}{4}$, the second quantity is larger. Hence $L$ is bounded above by
    \begin{align*}
        L & \leq s_1+s_2 \le a\left(\frac{1}{\cos\theta}+\frac{1}{\sin\theta}\right)= 2a\left(\frac{\cos\theta+\sin\theta}{\sin2\theta}\right). \label{eq:maxlen}
        \qedhere
    \end{align*}
\end{proof}

We also derive a general bound on the intersection length of $\ell$ with $A_a$ that is independent of $\theta$.

\begin{claim}\label{claim:line intersect}
For any line $\ell$,
   \[L:=\len(\ell\cap A_a)\le\frac{\sqrt{1-2a+2a^2}}{1-a}.\]
\end{claim}
\begin{proof}
    Let $\ell_0$ be the line that connects the bottom left corner $(0,0)$ of $Q$ and the bottom right corner  $(1-a,a)$ of the inner square $\overline{Q_a}$. The slope of $\ell_0$ is $s_0 = \frac{a}{1-a}$, and a simple computation yields
    \[\len(\ell_0\cap A_a) = \frac{\sqrt{1-2a+2a^2}}{1-a}.\]
    Note that $\ell_0$ does not intersect $Q_a$.

    %FIGURE HERE?
    
    We show that no line $\ell$ has longer intersection with $A_a$ than $\ell_0$. By symmetry, we may take $\ell$ to have slope $s\in [0,1]$. Observe that for the unit square $Q$, we have $\len(\ell\cap Q)\le \sqrt{1+s^2}$. Since $Q\supset A_a$, if $\len(\ell\cap A_a)> \len(\ell_0\cap A_a)$, then the slope of $\ell$ must exceed that of $\ell_0$. We henceforth assume $s\in (\frac{a}{1-a},1]$.

    Borrowing from the proof techniques in \Cref{claim:2}, if $\theta$ is the angle between $\ell$ and the $x$-axis, then 
    \[L = \len(\ell\cap A_a)\le \frac{a}{\sin\theta}+\frac{a}{\cos\theta}.\]
    Observe that if $\theta_0$ is the angle of $\ell_0$, then $\frac{\pi}{4}\ge \theta>\theta_0$ because $1\ge s>s_0$. Since the function $\frac{1}{\sin\theta}+\frac1{\cos\theta}$ is increasing away from $\frac{\pi}{4}$ and is symmetric about $\frac{\pi}{4}$, this leads to
    \begin{align*}
        \len(\ell\cap A_a) &\le \frac{a}{\sin\theta}+\frac{a}{\cos\theta}\\
         &\le \frac{a}{\sin\theta_0}+\frac{a}{\cos\theta_0}\\
         &=  \len(\ell_0\cap A_a).
    \end{align*}
    Notice the last line holds by construction of $\ell_0$; its intersections with the sides of $A_a$ achieve the full lengths $\frac{a}{\sin\theta_0},\frac{a}{\cos\theta_0}$ respectively. We conclude that no line has longer intersection with $A_a$.
\end{proof}

\begin{lemma}\label{lem:uncovered-area}
    Let $\mathcal{C}$ be a finite collection of planar planks $P_1,\dots,P_n$ such that $\onenorm{u_i} \ge \sqrt 2-\beta$ for each $i\in\{1,\dots,n\}$ for some fixed $\beta\in (0,\sqrt 2-1)$. Then for any $a\in(0,\frac12)$, the area in the collar $A_a$ covered by the planks in $\mathcal{C}$ is upper bounded as 
    \[\area(A_a \cap \mathcal{C}) \leq 2a\left(\frac{\sqrt2 - \beta}{1 - 2\sqrt2\beta + \beta^2}\right)w(\mathcal{C}).\]
\end{lemma}

\begin{proof}
    We estimate $\area(A_a \cap \C)$ by bounding the intersection of each plank with~$A_a$. We will perform our analysis for a single plank $P$ with midline $\ell$ and angle $\theta\in [0,\frac{\pi}{2}]$. Since $u$ is orthogonal to $\ell$, $\|u\|_1 = \cos\theta+\sin\theta$, so the hypothesis on $\mathcal{C}$ reads as $\cos\theta_i+\sin\theta_i \ge \sqrt2-\beta$ for each~$i$. Let $L := \len(\ell \cap A_a)$. 
    The upper bound on $L$ depends only on $\theta$. The area of the intersection of a plank with $A_a$ is obtained by integrating $L$ over translations of $\ell$. Therefore,  Claim~\ref{claim:2} yields:
    \begin{equation}\label{eq:maxar} \area(A_a \cap P) \leq 2a\left(\frac{\cos\theta+\sin\theta}{\sin2\theta}\right)w(P).\end{equation} 
    
    As a function of $\theta$, the right side of \eqref{eq:maxar} is symmetric about $\frac{\pi}{4}$ and is increasing away from $\frac{\pi}{4}$. Let $\theta^*\in [0,\frac{\pi}{4}]$ be the smallest angle such that a plank with angle $\theta^*$ satisfies $$\|u\|_1\ge \sqrt 2-\beta.$$ It follows that $\sin\theta^*+\cos\theta^* = \sqrt{2}-\beta$. We will now use the fact that substituting $\theta^*$ for $\theta$ does not decrease the right side of \eqref{eq:maxar}. Thus, we obtain the upper bound
    \begin{align*}
        \area(A_a \cap \mathcal{C}) &\leq \sum_{i=1}^n2a\left(\frac{\cos\theta_i+\sin\theta_i}{\sin2\theta_i}\right)w(P_i)\\
        &\leq 2a\left(\frac{\cos\theta^*+\sin\theta^*}{\sin2\theta^*}\right)w(\mathcal{C})\\
        &= 2a\left(\frac{\cos\theta^*+\sin\theta^*}{(\cos\theta^*+\sin\theta^*)^2-1}\right)w(\mathcal{C})\\
        &= 2a\left(\frac{\sqrt 2-\beta}{1-2\sqrt2 \beta+\beta^2}\right)w(\mathcal{C}).
    \end{align*}
    This proves the lemma.
\end{proof}

By invoking \Cref{claim:line intersect} instead of \Cref{claim:2} in the above argument, we derive the following statement analogously. 

\begin{lemma}\label{lem:uncovered-area-general}
    Let $\mathcal{C}$ be a finite collection of planar planks. Then for any $a\in(0,\frac12)$,
    \[\area(A_a \cap \mathcal{C}) \leq  \frac{\sqrt{1-2a+2a^2}}{1-a} w(\mathcal{C}).\]
\end{lemma}

We are ready to prove the main theorem in $\R^2$.

\begin{proof}[Proof of Theorem~\ref{thm:main} for $d=2$]
By \Cref{rk:open planks}, we may and do assume that the planks in the covering family are open.

Assume that for some $ 0 < \ee < 1/\sqrt2 - 1/2 $,  there exists a suitably chosen square $\delta Q \subset Q^\circ$ where $\delta := 1-1/\sqrt2+\ee < 1/2$ so that $Q \setminus \delta Q$ has a plank covering of total width less than~1. We will show that this results in a contradiction if $\eps = \eps_2$ for some $\eps_2>0$. By containment, this also yields the statement for every $\eps \in (0, \eps_2)$.

Accordingly, let $\C = \{ P_1, \ldots, P_n \}$ be a collection of open planks with $w(\C) <1$ that covers  $Q \setminus \delta Q$.     For each $i$, let  $v_i = (x_i,y_i)$ be the cross-normal vector of $P_i$, and $u_i = v_i / \|v_i \|$ be its unit normal.

Consider $X_\C$, the Bang point set of $\C$ as constructed in \Cref{lem:bang}. By construction, its width and height equal the sums of the lengths of the horizontal and vertical projections of the $v_i$, respectively; both are at most $w(\C)<1$. On the other hand, by \Cref{lemma:new plank}, each coordinate width exceeds $1-\delta$: otherwise, $X_\C$ could be translated into one of the rectangles $R_1,R_2$. Therefore, 
    \[\sum_{i=1}^n |x_i| \geq 1-\delta \quad\text{and}\quad \sum_{i=1}^n |y_i|\geq 1-\delta.\]
Add these two inequalities and observe that the left side becomes $\sum_{i=1}^n  \onenorm{v_i}$. Thus,
    \begin{equation}
        \label{eq:diff}
        \sum_{i=1}^n  \|v_i\|_1 \geq 2-2\dd = \sqrt{2}-2\ee.
    \end{equation}
Combined with $w(\C)<1$ and \Cref{claim:max1norm}, the above inequality expresses that most of the total plank width is carried by nearly diagonal planks.

Next, let $a \in (0,  1 - 1/\sqrt2)$ be a parameter whose value we will fix later. Recall that $A_a$ is a collar (i.e., centered annulus) with thickness $a$, as defined in \eqref{eq:Adef}. Note that its central hole $Q_a$ is (typically) different from the cut-out hole $\delta Q$, see \Cref{fig:collar-with-hole}. 

\begin{figure}[h]
\centering
\begin{tikzpicture}[scale=3.6]

  % Outer square Q
  \coordinate (A) at (-1,-1);
  \coordinate (B) at ( 1,-1);
  \coordinate (C) at ( 1, 1);
  \coordinate (D) at (-1, 1);

  % Centered inner square (1-2a)_o Q
  \coordinate (Ai) at (-0.42,-0.42);
  \coordinate (Bi) at ( 0.42,-0.42);
  \coordinate (Ci) at ( 0.42, 0.42);
  \coordinate (Di) at (-0.42, 0.42);

  % Shifted square delta Q
  \coordinate (Ad) at (0.22,0.08);
  \coordinate (Bd) at (0.92,0.08);
  \coordinate (Cd) at (0.92,0.78);
  \coordinate (Dd) at (0.22,0.78);

  % Fill all of Q with medium gray
  \fill[gray!64] (A)--(B)--(C)--(D)--cycle;

  % Fill centered inner square with lighter gray
  \fill[gray!22] (Ai)--(Bi)--(Ci)--(Di)--cycle;

  % Fill delta Q with white, so overlap stays white
  \fill[white] (Ad)--(Bd)--(Cd)--(Dd)--cycle;

  \fill[gray!64] (D)--(Di)--(Ci)--(C)--cycle;
  \fill[gray!64] (C)--(Ci)--(Bi)--(B)--cycle;

  \fill[gray!42] (Ad)--(Bd)--(Cd)--(Dd)--cycle;

  \fill[white] (0.22,0.08)--(0.42,0.08)--(0.42,0.42)--(0.22,0.42)--cycle;

  % -------------------------------------------------
  % Planks (clipped to Q)
  % -------------------------------------------------
  \begin{scope}
    \clip (A)--(B)--(C)--(D)--cycle;

    % Blue plank along the opposite diagonal
    \fill[blue!20, opacity=.35]
      (-1.18,1.2) -- (-0.92,1.18) -- (1.18,-0.92) -- (1.2,-1.18) -- cycle;

    \fill[green!20, opacity=.35]
      (-1.15,-0.22) rectangle (1.15,0.02);

    % Blue plank dashed boundaries
    \draw[blue!80, dashed, line width=0.8pt]
      (-1.18,1.2) -- (1.2,-1.18);
    \draw[blue!80, dashed, line width=0.8pt]
      (-0.92,1.18) -- (1.18,-0.92);

    % Horizontal plank dashed boundaries
    \draw[green!70!black, dashed, line width=0.8pt]
      (-1.15,-0.22) -- (1.15,-0.22);
    \draw[green!70!black, dashed, line width=0.8pt]
      (-1.15,0.02) -- (1.15,0.02);
  \end{scope}

  % -------------------------------------------------
  % Draw square boundaries
  % -------------------------------------------------
  \draw[thick] (A)--(B)--(C)--(D)--cycle;
  \draw[thick] (Ai)--(Bi)--(Ci)--(Di)--cycle;
  \draw[thick] (Ad)--(Bd)--(Cd)--(Dd)--cycle;

  % -------------------------------------------------
  % Labels inside diagram
  % -------------------------------------------------
  \node at (0,-0.69) {$A_a$};
  \node at (0.00,-0.34) {$Q_a$};
  \node at (0.32,0.70) {$\delta Q$};
  \node[anchor=south east] at (D) {$Q$};

\end{tikzpicture}
\caption{The collar $A_a$ overlaid on the square annulus $Q\setminus \delta Q$. Also shown is a diagonal plank (in blue) and a non-diagonal plank (in green).}
\label{fig:collar-with-hole}
\end{figure}
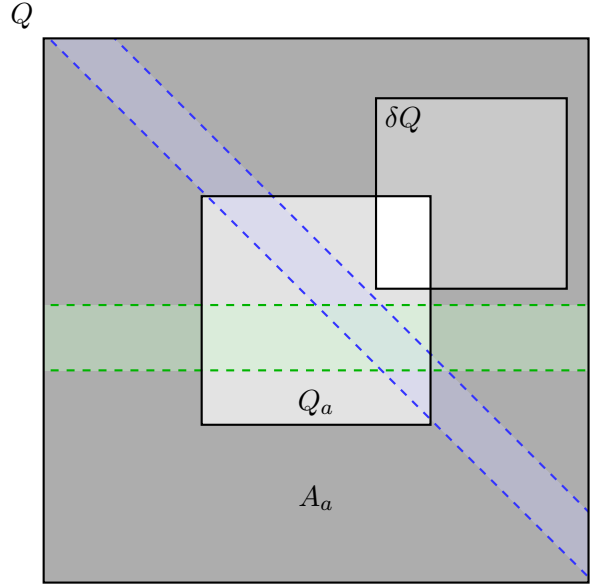

\noindent 
Our goal is to show, by area comparison, that for a suitable choice of $\eps$ and $a$, the annulus $A_a$ cannot be covered by the union of planks in $\C$ and the square $\delta Q$, which leads to a contradiction. 

In order to exploit the nearly diagonal structure of $\C$, we introduce yet another parameter $\alpha>1$, to be fixed later, that satisfies 
\begin{equation}\label{eq:alphaconstraint}
2 \alpha \eps < \sqrt{2} - 1.
\end{equation}
We partition the planks into two classes: diagonal ($\D$) and non-diagonal ($\N$), based on $\alpha$ and the 1-norm of their unit normal vector:
\begin{align*}
    \mathcal D &\coloneqq \{P_i : \onenorm{u_i} \geq \sqrt 2 - 2\alpha\ee\}\\
    \mathcal N &\coloneqq \{P_i : \onenorm{u_i} < \sqrt 2 - 2\alpha\ee\},
\end{align*}
so that $\C = \D \cup \N$. Let $w(\D),w(\N)$ refer to the total width of planks in $\D,\N$ respectively. Then 
\begin{equation}\label{eq: wD wN}
 w(\D) + w(\N) = w(\C) <1.
\end{equation}
As a consequence of \eqref{eq:diff}, the total width of the planks in $\N$ must be small. We express this property precisely with a quantitative estimate.
\begin{claim}\label{claim: wN}
    $w(\mathcal{N})\le \frac{1}{\alpha}$.
\end{claim}

\begin{proof}
Using \eqref{eq: wD wN}, \Cref{claim:max1norm},  and the definition of $\mathcal N$,
    \begin{align*}
    \sqrt 2-w(\mathcal{N}) 2\alpha\ee &\geq  w(\C)\sqrt 2 - w(\mathcal N) 2\alpha \ee\\
    &= w(\mathcal{D})\sqrt 2 + w(\mathcal{N})(\sqrt{2}-2\alpha\ee)\\
    &\ge \sum_{i=1}^n w(P_i) \|u_i\|_1 \\
    &\ge \sqrt 2 -2\ee
    \end{align*}
    The last line follows from \eqref{eq:diff}. This readily implies that $w(\mathcal{N})\le \frac{1}{\alpha}$.
\end{proof}

In order to show that $A_a$ cannot be covered, we produce estimates for $\area(A_a \cap \mathcal{D})$, $\area(A_a \cap \mathcal{N})$ and $\area(A_a \cap \delta Q)$ separately.  
 
 First, applying \Cref{lem:uncovered-area} with $\beta = 2\alpha\ee$ (which is permitted by \eqref{eq:alphaconstraint}) leads directly to the bound
\begin{equation}\label{eq:areaD}
    \area(A_a \cap \mathcal{D}) \le \frac{2a(\sqrt2 - 2\alpha\ee)}{1 - 4\sqrt2\alpha\ee + 4\alpha^2\ee^2} w(\mathcal{D}) \leq 2 a \left( \sqrt{2} + \frac {13} 2 \alpha \ee \right)  w(\mathcal{D})
\end{equation}
which holds whenever 
\begin{equation}\label{eq:ae_criterion}
    2 \alpha \eps \leq 0.03,
\end{equation}
since for $0 \leq x \leq 0.03$ we have 
\[
\frac{\sqrt2 - x}{1 - 2\sqrt2 x + x^2} \leq \sqrt{2} + \frac{13}{4}x.
\]
For shorthand, we will write $\lambda = \sqrt{2} + \frac {13} 2 \alpha \ee$, thus $\area(A_a \cap \mathcal{D}) \le 2 a \lambda w(\mathcal{D})$. 

Second, \Cref{lem:uncovered-area-general} implies that
\begin{equation}\label{eq:areaN}
    \area(A_a \cap \mathcal{N}) \le \frac{\sqrt{1-2a+2a^2}}{1-a}w(\mathcal{N}) \leq (1 + a^2) w(\mathcal{N})
\end{equation}
since the condition  $a <1 - 1/\sqrt{2} < 1/\sqrt{2}$ guarantees that 
\[
\frac{\sqrt{1-2a+2a^2}}{1-a} = \sqrt{1 + \frac{a^2}{(1-a)^2}} \leq 1 + \frac{a^2}{2(1-a)^2} \leq 1 + a^2.
\]

Finally, observe that since  $a < 1 - 1/\sqrt{2} < \dd$, the overlap between $A_a$ and the square $\delta Q$ is bounded above by the corner-position value. 
Accordingly,
\begin{equation}\label{eq:areaQ}
  \area(A_a \cap \delta Q) \le \delta^2-(\delta-a)^2 = 2a\delta - a^2.
\end{equation}

When finding suitable values for $\eps, a$ and $\alpha$, we will also require that
\begin{equation}\label{eq:aepsineq}
 2 a \lambda  <  1+a^2.
\end{equation}

Combining the area bounds \eqref{eq:areaD}, \eqref{eq:areaN} and \eqref{eq:areaQ} with \eqref{eq: wD wN}, \eqref{eq:aepsineq} and \Cref{claim: wN}, we derive that the total area in $A_a$ covered by $\C \cup \delta Q$ can be estimated from above as 
\begin{align*}
\area&(A_a \cap (\C \cup \delta Q)) \\
&\leq \area(A_a \cap \D) +  \area(A_a \cap \N) +  \area(A_a \cap \delta Q) 
\\ 
&\leq  2 a\lambda   w(\mathcal{D}) + (1+a^2)w(\mathcal{N}) +  2a\delta - a^2 \\
&<  2 a\lambda   (1 - w(\mathcal{N})) + (1+a^2)w(\mathcal{N}) +  2a\delta - a^2 \\
&= 2 a\lambda  + \left((1+a^2) -  2 a\lambda  \right) w(\mathcal{N}) +  2a\delta - a^2\\
&\leq 2 a\lambda  + \left((1+a^2) -  2 a\lambda  \right) \cdot \frac 1 \alpha+  2a\delta - a^2\\
& = 
(2+\sqrt{2})a-a^2
+\frac{1+a^2-2\sqrt{2}\,a}{\alpha}
+a(13\alpha-11)\ee.
\end{align*}
Note that 
\[
\area(A_a) = 1 - (1-2a)^2 = 4a-4a^2.
\]
Therefore, the area of $A_a$ left uncovered by $\mathcal C\cup\delta Q$ is greater than
\begin{align*}
U(\ee, a, \alpha) &:= 4a-4a^2-\left((2+\sqrt{2})a-a^2
+\frac{1+a^2-2\sqrt{2}\,a}{\alpha}
+a(13\alpha-11)\ee \right) \\
&=
(2-\sqrt{2})a-3a^2
+\frac{2\sqrt{2}\,a-1-a^2}{\alpha}
-a(13\alpha-11)\ee.
\end{align*}

In order to complete the proof, it suffices to find a specific triple $(\ee, a, \alpha)$ with $\eps \in (0,  1/\sqrt2 - 1/2)$, $a \in (0,  1 - 1/\sqrt2)$, and $\alpha >1 $ satisfying \eqref{eq:ae_criterion} and \eqref{eq:aepsineq}, such that
\[U(\ee, a,\alpha)>0.\]

Note that $U(\ee, a, \alpha)$ is linear in $\eps$, whose coefficient is negative.  Thus $U(\ee,a,\alpha)$ is strictly decreasing in $\ee$.
Consequently, for fixed $a$ and $\alpha$, the condition $U(\ee,a,\alpha)>0$
is equivalent to
\[
\ee<E(a,\alpha),
\]
where
\begin{align*}
E(a,\alpha)
&:=
\frac{
(2-\sqrt{2})a-3a^2+
\frac{2\sqrt{2}\,a-1-a^2}{\alpha}
}{
a(13\alpha-11)
} \\
&=
\frac{1}{13\alpha-11}
\left(
2-\sqrt{2}-3a+
\frac{2\sqrt{2}-a-1/a } {\alpha}
\right).
\end{align*}

In order to find a suitable $\eps$, we maximize $E(a, \alpha)$ subject to the constraints  $a \in (0,  1 - 1/\sqrt2)$, and $\alpha >1 $ and then check that the solutions also satisfy \eqref{eq:ae_criterion} and \eqref{eq:aepsineq}. 
For fixed $\alpha$, differentiating $E(a, \alpha)$ with respect to $a$ yields that the 
unconstrained maximum over \(a>0\) is attained at
\[
a=\frac{1}{\sqrt{3\alpha+1}}.
\]
Note that for $\alpha > (5 + 4 \sqrt{2})/3 \approx 3.552 $, this maximizing value lies in the required interval
$(0,1-1/\sqrt2)$.
Substituting this value back in $E(a, \alpha)$, the function to be maximized becomes
\[
\varphi(s)
=
\frac{
3\bigl((2-\sqrt{2})(s^2-1)+6(\sqrt{2}-s)\bigr)
}{
(s^2-1)(13s^2-46)
}
\]
where $s = \sqrt{3 \alpha +1}$. A direct computation shows that $\varphi(s)>0$ for $s >  (1 + 1/ \sqrt{2})(3 +\sqrt{27 - 16 \sqrt{2}}) \approx 8.691$, and its relevant critical point numerically is $
s_*\approx13.217911$ that corresponds to $
\alpha_*\approx57.904393$, $a_*\approx0.075655$, and $\varphi(s_*) = E(a_*,\alpha_*) \approx2.400925\times10^{-4}$.

For a convenient explicit choice, take
\[
\alpha=58,\qquad
a=0.0756=\frac{189}{2500},
\qquad
\ee=2.4\times10^{-4}=\frac{3}{12500}.
\]
Then
\[
2\alpha\ee=0.02784<0.03,
\qquad
2a\left(\sqrt2+\frac{13}{2}\alpha\ee\right)
\approx0.22751<1+a^2,
\]
so both \eqref{eq:ae_criterion} and \eqref{eq:aepsineq} hold. A direct substitution gives
\[
U\left(\frac{3}{12500},\frac{189}{2500},58\right)
=
\frac{6452063}{62500000}
-\frac{1323}{18125}\sqrt{2} >0.
\]
Therefore, the argument proves \Cref{thm:main} for $d=2$ with
\[
\ee_2=2.4\times10^{-4}.
\qedhere
\]
\end{proof}

\section{An Improved Bound in $\R^3$}\label{sec:improved-bound 3d}

In this section, we prove \Cref{thm:main} for $d=3$. The argument follows the same structure as in the planar case. We again divide the covering planks into nearly-diagonal and non-diagonal classes, and estimate the volume of a thin centered collar covered by these two families and by the hole $\delta Q$, which will lead to a contradiction for sufficiently small $\eps$. The corresponding volume estimates will be carried out via a reduction to two-dimensional coordinate slices.

Throughout the section, we work with $d=3$. Consequently, $Q$ and $A_a$ denote the $3$-dimensional cube and collar, respectively.

Similarly to the planar case, `diagonality' of a vector $v$ is measured by the proportion $\| v\|_1 / \| v \|_2$. A nearly-diagonal vector must have coordinates whose absolute values are close to each other. In order to quantify this property, we prove a stability version of \Cref{claim:max1norm}.

\begin{claim}\label{claim:min and max}
Assume that $u = (u_x,u_y,u_z) \in \R^3$ is a unit vector that satisfies $\|u\|_1\ge \sqrt3 - \beta$ with some $\beta\in [0,\sqrt 3 -\sqrt 2)$. Then 
    \begin{align*}
        \max \{|u_x|,|u_y|,|u_z|\} &\le \frac{1}{\sqrt 3}-\frac{\beta}{3}+ \frac{\sqrt{4\sqrt3 \beta -2\beta^2}}{3}:= M(\beta).
    \end{align*}
\end{claim}

\begin{proof}
Let $s=\sqrt3-\beta>1$ and $x=|u_x|$. By the Cauchy--Schwarz inequality,
\[
s\le x+|u_y|+|u_z|
\le x+\sqrt{2(1-x^2)}.
\]
Since $x\le1<s$, squaring gives
\[
3x^2-2sx+s^2-2\le0.
\]
Consequently,
\[
x\le\frac{s+\sqrt{6-2s^2}}{3}
=\frac{\sqrt3-\beta+\sqrt{4\sqrt3\beta-2\beta^2}}{3}
=M(\beta).
\]
By symmetry, the same bound holds for $|u_y|$ and $|u_z|$.
\end{proof}
\noindent
We note that $M(\beta)$ is increasing for $\beta\in [0,\sqrt 3 -\sqrt 2)$ with $M(0) = 1/\sqrt{3}$.

We will estimate the volume of a collar covered by a given plank by integrating the area of the corresponding two-dimensional coordinate slices. Let $P$ be a 3-dimensional plank of the form
\[
P = \Big \{(x,y,z) \in \R^3: \ |\langle (x,y,z), u \rangle - t| \leq \frac w 2  \Big \}
\]
for a unit normal vector $u = (u_x, u_y, u_z) \in \R^3$, width $w >0$, and shift $t \in \R$. Let $H \subset \R^3$ be a translate of a coordinate plane, that is, a plane orthogonal to a standard basis vector. Without loss of generality, assume that $H$ is perpendicular to the $z$-axis. Suppose that $u$ is not parallel to the $z$-axis, that is, $u_z^2 <1$. Then $H \cap P$ is a 2-dimensional plank in $H$. To calculate its normal vector and width, assume that $H = \{(x,y,z_0) \}$ for $x,y \in \R$ and a fixed $z_0 \in \R$. Then 
\[
H \cap P = \Big \{(x,y,z_0) \in \R^3: \ |\langle (x,y,z_0), u \rangle - t| \leq \frac w 2  \Big \}.
\]
Identifying $H$ with $\R^2$, we obtain
\begin{align*}
H \cap P&= \Big \{(x,y) \in \R^2: \ |\langle (x,y), (u_x, u_y) \rangle - ( t - z_0 \cdot u_z)| \leq \frac w 2  \Big \} .
\end{align*}
Thus, $H \cap P$ is a 2-dimensional plank with unit normal 
\begin{equation}\label{eq:uhp0}
u(H \cap P) = \frac{(u_x, u_y)}{(u_x^2+ u_y^2)^{1/2}} = \frac{(u_x, u_y)}{(1-u_z^2)^{1/2}}
\end{equation}
and width 
\begin{equation}\label{eq:whp0}
w(H \cap P) = \frac w {(u_x^2+ u_y^2)^{1/2}} = \frac{w}{(1-u_z^2)^{1/2}}\,.
\end{equation}

Combining the above formulae with the estimates of \Cref{claim:min and max}, we readily obtain the following bounds that show that nearly-diagonal $3$-dimensional planks descend to nearly-diagonal $2$-dimensional planks. 

\begin{lemma}\label{lem:diagonal after intersection}
     Let $P$ be a $3$-dimensional plank of width $w$ and unit normal $u$ that satisfies $\|u\|_1\ge \sqrt{3}-\beta$ with some $\beta\in (0,\sqrt 3-\sqrt 2)$. Let $H$ be a translate of a coordinate plane. Then $H\cap P$ is a $2$-dimensional plank with unit normal $u(H \cap P) \in \R^2$ and width $w(H\cap P) $ that satisfy
     \begin{equation}\label{eq:uhp}
\|u(H\cap P)\|_1\ge \sqrt 2  - \frac 3 2 \beta
     \end{equation}
     and 
\begin{equation}\label{eq:whp}
w(H\cap P) \le \frac{w}{\sqrt{1-M(\beta)^2}}.
\end{equation}
 \end{lemma}

\begin{proof}
Without loss of generality, we may assume that $H$ is perpendicular to the $z$-axis. 
To obtain \eqref{eq:uhp}, note that 
\[
|u_x| + |u_y| \geq \sqrt{3}-\beta - |u_z|.
\]
Therefore, by \eqref{eq:uhp0},
\[
\| u(H \cap P )\|_1 \geq \frac{\sqrt{3}-\beta - |u_z|}{(1-u_z^2)^{1/2}}.
\]
Since $\beta < \sqrt{3}-1$, the function
\[
\psi(s) = \frac{\sqrt{3}-\beta - s}{ (1 - s^2)^{1/2}}
\]
attains its minimum on $(0,1)$ at $s = 1/ (\sqrt{3} - \beta)$ with minimum value
\[
\sqrt{(\sqrt{3}- \beta)^2 -1} = \sqrt{2 - 2 \sqrt{3} \beta + \beta^2}.
\]
Note that the function on the right hand side is strictly concave. Therefore, on the admissible interval $\beta \in (0, \sqrt{3} - \sqrt{2})$ it is lower bounded by the chord connecting the extreme points of its graph: 
\[
 \sqrt{2 - 2 \sqrt{3} \beta + \beta^2}  > \sqrt{2} - \frac{1 - \sqrt{2}}{\sqrt{2} - \sqrt{3}} \beta >  \sqrt{2} - \frac 3 2 \beta, 
\]
which completes the proof of \eqref{eq:uhp}.

The estimate \eqref{eq:whp} follows readily from  \eqref{eq:whp0} and the fact $|u_z| \leq M(\beta)$ provided by~\Cref{claim:min and max}.
\end{proof}

Next, we derive the 3-dimensional analogue of~\Cref{lem:uncovered-area}. Just as in the planar case, we fix a thickness parameter $0 < a < 1/2$ and consider the centered annulus $A_a$ defined in \eqref{eq:Adef}. We wish to control the volume of $A_a$ covered by a collection of nearly diagonal planks. For $\beta \in [0, \frac 2 3 (\sqrt2-1))$, define 
\[
D(\beta):=\left(\frac{\sqrt2 - \frac 3 2 \beta}{1 - 3\sqrt2\beta + \frac 9 4 \beta^2}\right) \cdot \frac 1 {\sqrt{1-M(\beta)^2}}.
\]
For later reference, we note that $D(\beta)$ is continuous and increasing throughout its domain (as is seen by analyzing its factors) and
\begin{equation}\label{eq:betazero}
    D(0)=\sqrt3.
\end{equation}

\begin{lemma}\label{lemma:uncovered-volume}
    Let $\mathcal{C}$ be a collection of 3-dimensional planks $P_1,\dots,P_n$ such that $\onenorm{u_i} \ge \sqrt 3-\beta$ for each $i\in\{1,\dots,n\}$ for some fixed $\beta\in (0,\frac 2 3 (\sqrt2-1))$. Then for any $a\in(0,\frac12)$, the volume in the 3-collar $A_a$ covered by the planks in $\mathcal{C}$ is upper bounded as 
    \[
     \vol(A_a\cap \mathcal{C}) \le  3a \, D(\beta) w(\mathcal{C}).
    \]
\end{lemma}

\begin{proof}
   We reduce the problem to the planar setting.  Let $A_{xy}= [0,1]^2 \setminus (a,1-a)^2$ be  the relevant annulus in the $xy$-plane that arises as the intersection of $A_a$ with the horizontal plane at height $1/2$. We will first estimate the volume of the vertical tube $A_{xy}\times[0,1]$ covered by members of $\mathcal{C}$, see~\Cref{fig:xy-slice-cubic-annulus-plank}. We give a bound by integrating the covered area at height $z$ for $z \in [0,1]$. 
   
   Accordingly, fix any height $z\in [0,1]$, and let $H$ be the $xy$-plane at height $z$. Then $H \cap (A_{xy}\times[0,1]) = A_{xy}\times \{z\}$. Lemma~\ref{lem:diagonal after intersection} shows that for each plank $P \in \mathcal{C}$, the intersection $H \cap P$ is a planar plank of width
   \[
   w(H\cap P)\le \frac{w(P)}{\sqrt{1-M(\beta)^2}}
   \]
   and normal vector $u(H\cap P)$ that satisfies
   \[
   \|u(H\cap P)\|_1\ge \sqrt 2  - \frac 3 2 \beta.
   \]
Therefore, by \Cref{lem:uncovered-area},
    \begin{align*}
        \area\Big[(A_{xy}\times \{z\})\cap \mathcal{C}\Big] &\le 2a \left(\frac{\sqrt2 - \frac 3 2 \beta}{1 - 3\sqrt2\beta + \frac 9 4 \beta^2}\right) \cdot \frac 1 {\sqrt{1-M(\beta)^2}} \cdot w(\mathcal{C}) \\
        &= 2 a \, D(\beta)w(\mathcal{C}).
    \end{align*}
By integrating with respect to the height of the section, we derive that the volume of $A_{xy}\times[0,1]$ covered by members of $\mathcal{C}$ is at most 
 \begin{align*}
        I_{xy} &= \int_0^1 \area[(A_{xy}\times \{z\})\cap \mathcal{C}] \ dz\\
        &\le 2 a \, D(\beta) w(\mathcal{C}).
    \end{align*}
By symmetry, the same estimate holds for the volumes $I_{xz}$ and $I_{yz}$ covered by members of $\mathcal{C}$ in the analogous $xz$- and $yz$-tubes, respectively. Finally, notice that  each point of $A_a$ is covered at least twice by these three tubes. We may now bound $\vol(A_a\cap \mathcal{C})$ as
    \begin{align*}
    \vol(A_a\cap \C) &\le \frac12\Big[I_{xy}+I_{xz}+I_{yz}\Big]\\
    &\le 3 a \, D(\beta) w(\mathcal{C}).
    \end{align*}
    This proves the lemma.
\end{proof}

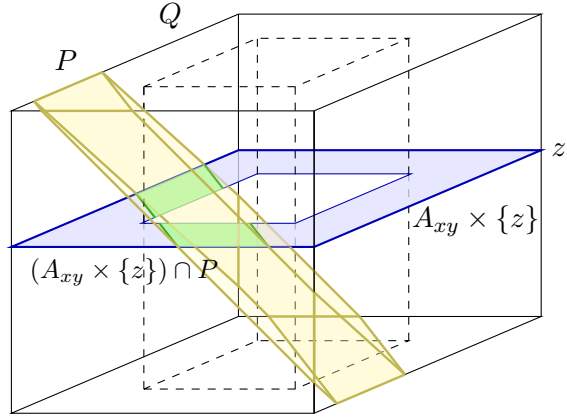
\begin{figure}[H]
\centering
\begin{tikzpicture}[
    scale=2,
    x={(1cm,0cm)},
    y={(0.75cm,0.32cm)},
    z={(0cm,1cm)}
]

%-----------------------------------
% Parameters
%-----------------------------------
\def\B{2}
\def\C{0.5}
\def\D{1.5}
\def\Z{1.1} % slicing height

%===================================
% Coordinates for diagonal plank P
% bounded by x+y+z = 2.2 and x+y+z = 2.8
% so cross normal is parallel to (1,1,1)
%===================================

% Hexagon on plane x+y+z = 2.2
\coordinate (A1) at (0,0.2,2);
\coordinate (B1) at (0.2,0,2);
\coordinate (C1) at (2,0,0.2);
\coordinate (D1) at (2,0.2,0);
\coordinate (E1) at (0.2,2,0);
\coordinate (F1) at (0,2,0.2);

% Hexagon on plane x+y+z = 2.8
\coordinate (A2) at (0,0.8,2);
\coordinate (B2) at (0.8,0,2);
\coordinate (C2) at (2,0,0.8);
\coordinate (D2) at (2,0.8,0);
\coordinate (E2) at (0.8,2,0);
\coordinate (F2) at (0,2,0.8);

%-----------------------------------
% Outer cube
%-----------------------------------
\draw (0,0,0) -- (\B,0,0) -- (\B,\B,0) -- (0,\B,0) -- cycle;
\draw (0,0,\B) -- (\B,0,\B) -- (\B,\B,\B) -- (0,\B,\B) -- cycle;
\draw (0,0,0) -- (0,0,\B);
\draw (\B,0,0) -- (\B,0,\B);
\draw (\B,\B,0) -- (\B,\B,\B);
\draw (0,\B,0) -- (0,\B,\B);

%-----------------------------------
% Inner cube [a,1-a]^3
%-----------------------------------
\draw[dashed] (\C,\C,0) -- (\D,\C,0) -- (\D,\D,0) -- (\C,\D,0) -- cycle;
\draw[dashed] (\C,\C,\B) -- (\D,\C,\B) -- (\D,\D,\B) -- (\C,\D,\B) -- cycle;
\draw[dashed] (\C,\C,0) -- (\C,\C,\Z);
\draw[dashed] (\D,\C,0) -- (\D,\C,\Z);
\draw[dashed] (\D,\D,0) -- (\D,\D,\Z);
\draw[dashed] (\C,\D,0) -- (\C,\D,\Z);

%-----------------------------------
% Horizontal slice z = Z : A_xy x {z}
%-----------------------------------
\fill[blue!12, opacity=.75]
  (0,0,\Z) -- (\B,0,\Z) -- (\B,\B,\Z) -- (0,\B,\Z) -- cycle;
\draw[dashed] (\C,\C,\B) -- (\C,\C,\Z);

%-----------------------------------
% Slice outlines
%-----------------------------------
\draw[thick, blue!70!black]
  (0,0,\Z) -- (\B,0,\Z) -- (\B,\B,\Z) -- (0,\B,\Z) -- cycle;

\draw[thick, blue!70!black]
  (\C,\C,\Z) -- (\D,\C,\Z) -- (\D,\D,\Z) -- (\C,\D,\Z) -- cycle;

  %-----------------------------------
% Green intersection of plank with blue slice
% Since z = 1.1, strip is 1.1 <= x+y <= 1.7
%-----------------------------------
\fill[green!45, opacity=.70]
  (0,1.1,\Z) --
  (0,1.7,\Z) --
  (1.7,0,\Z) --
  (1.1,0,\Z) -- cycle;

  %-----------------------------------
% Green strip boundary lines
%-----------------------------------
\draw[green!60!black, thick]
  (0,1.1,\Z) -- (1.1,0,\Z);

\draw[green!60!black, thick]
  (0,1.7,\Z) -- (1.7,0,\Z);

%-----------------------------------
% Mask inner square to turn slice into annulus
%-----------------------------------
\fill[white, opacity=1]
  (\C,\C,\Z) -- (\D,\C,\Z) -- (\D,\D,\Z) -- (\C,\D,\Z) -- cycle;

%OVERLAY DOTTED GRAY
\draw[dashed] (\C,\C,\Z) -- (\C,\C,\B);
\draw[dashed] (\D,\C,\B) -- (\D,\C,\Z);
\draw[dashed] (\D,\D,\Z) -- (\D,\D,\B);
\draw[dashed] (\C,\D,\Z) -- (\C,\D,\B);
\draw[dashed] (\C,\D,\Z-.04) -- (\C,\D,\Z-.32);

\draw (\B,0,0) -- (\B,0,\B);
\draw (0,\B,\Z-0.212) -- (0,\B,\Z-.48);

%----

%-----------------------------------
% Fill the 3D plank P (yellow fills only)
%-----------------------------------
\fill[yellow!45, opacity=.22] (A1)--(B1)--(B2)--(A2)--cycle;
\fill[yellow!45, opacity=.22] (B1)--(C1)--(C2)--(B2)--cycle;
\fill[yellow!45, opacity=.22] (C1)--(D1)--(D2)--(C2)--cycle;
\fill[yellow!45, opacity=.22] (D1)--(E1)--(E2)--(D2)--cycle;
\fill[yellow!45, opacity=.22] (E1)--(F1)--(F2)--(E2)--cycle;
\fill[yellow!45, opacity=.22] (F1)--(A1)--(A2)--(F2)--cycle;
\fill[yellow!45, opacity=.14] (A2)--(B2)--(C2)--(D2)--(E2)--(F2)--cycle;

\draw[yellow!70!black, line width=0.9pt] (A1)--(B1);
\draw[yellow!70!black, line width=0.9pt] (B1)--(C1);
\draw[yellow!70!black, line width=0.9pt] (C1)--(D1);
\draw[yellow!70!black, line width=0.9pt] (D1)--(E1);
\draw[yellow!70!black, line width=0.9pt] (E1)--(F1);

\draw[yellow!70!black, line width=0.9pt] (0,1.1,1.1)--(0,0.2,2);
\draw[green!30!orange!80!yellow!60, line width=0.9pt, opacity=0.5] (0,1.1,1.1)--(0,1.31,0.89);
\draw[yellow!70!black, line width=0.9pt] (0,1.31,0.89)--(0,1.386,0.814);
\draw[green!30!orange!80!yellow!60, line width=0.9pt, opacity=0.5] (0,1.625,0.575)--(0,1.386,0.814);
\draw[yellow!70!black, line width=0.9pt] (0,1.625,0.575)--(F1);

\draw[yellow!70!black, line width=0.9pt] (A2)--(B2);
\draw[yellow!70!black, line width=0.9pt] (B2)--(C2);
\draw[yellow!70!black, line width=0.9pt] (C2)--(D2);
\draw[yellow!70!black, line width=0.9pt] (D2)--(E2);

\draw[yellow!70!black, line width=0.9pt] (0.185,2,0.615)--(0,2,0.8);
\draw[yellow!70!blue, line width=0.9pt, opacity=0.4] (0.185,2,0.615)--(0.34,2,0.46);
\draw[yellow!70!black, line width=0.9pt] (E2)--(0.34,2,0.46);

\draw[yellow!70!black, line width=0.9pt] (0,1.7,\Z)--(A2);
\draw[yellow!70!blue, line width=0.9pt, opacity=0.4] (0,1.7,1.1)--(0,1.91,0.89);
\draw[yellow!70!black, line width=0.9pt] (0,1.91,0.89)--(F2);

\draw[yellow!70!black, line width=0.9pt] (A1)--(A2);
\draw[yellow!70!black, line width=0.9pt] (B1)--(B2);
\draw[yellow!70!black, line width=0.9pt] (C1)--(C2);
\draw[yellow!70!black, line width=0.9pt] (D1)--(D2);
\draw[yellow!70!black, line width=0.9pt] (E1)--(E2);
\draw[yellow!70!black, line width=0.9pt] (F1)--(F2);

\draw[yellow!70!green!45!white!70, line width=0.9pt, opacity=0.8] (0,2,.4675)--(0,2,.614);

%-----------------------------------
% Height guide
%-----------------------------------
\node[right] at (\B,\B,\Z) {$z$};

\draw (\B,0,0) -- (\B,0,\B);
\draw[yellow!70!black, line width=0.9pt] (C1)--(C2);

%---

%-----------------------------------
% Labels
%-----------------------------------
\node at (2.228,1.12,\Z-0.185) {$A_{xy}\times\{z\}$};
\node at (0,0.48,\B+0.18) {$P$};
\node at (.31,0.58,\Z-.35) {\small $(A_{xy}\times\{z\})\cap P$};
\node at (0,1.4,\B+0.18) {$Q$};

\end{tikzpicture}
\caption{A diagonal plank through $Q$ is depicted in yellow; the green region is its intersection with the blue annular slice $A_{xy}\times\{z\}$.}
\label{fig:xy-slice-cubic-annulus-plank}
\end{figure}

Our final ingredient is the spatial analogue of \Cref{lem:uncovered-area-general}. For simplicity, we use a coarse estimate, which still suffices for proving the \emph{existence} of $\ee_3>0$ in Theorem~\ref{thm:main}.

\begin{lemma}\label{lem:uncovered-volume-general}
    Let $\mathcal{C}$ be a finite collection of 3-dimensional planks. Then for any $a\in(0,\frac12)$,
    \[\vol(A_a \cap \mathcal{C}) \leq  \sqrt{2} w(\mathcal{C}).\]
\end{lemma}

\begin{proof}
It is a standard fact that the largest area of intersection between a plane and the unit cube is at most $\sqrt 2$ (see \cite{Ball1986}). By integration across the width of a plank, it follows that for any $3$-dimensional plank $P$,
\[\vol(A_a\cap P)\le \sqrt 2w(P).
\]
Summing over $P \in \C$ proves the claim.
\end{proof}

We are ready to establish the proof of the main result in the $3$-dimensional setting.

\begin{proof}[Proof of Theorem~\ref{thm:main} for $d=3$]
By Remark~\ref{rk:open planks}, it suffices to consider open planks. 
Write $\delta_0 = 1 - 1/\sqrt{3}$.
Suppose, toward a contradiction, that for some $\eps > 0$, there is a suitably placed cube $\delta Q \subset Q^\circ$ with 
\[
\delta= \delta_0 + \eps = 1  - \frac 1 {\sqrt 3}+\eps 
\]
so that the annulus $Q\setminus\delta Q$ has a plank covering $\mathcal C=\{P_1,\dots,P_n\}$
with $w(\mathcal C)<1$. As before, let $v_i$ and $u_i$ denote the cross-normal vector and unit normal of $P_i$, respectively.

By Lemma~\ref{lemma:new plank}, the Bang point set $X_{\mathcal C}$
has width at least $1-\delta$ in each coordinate direction. Hence, similarly to the planar case, 
\[
\sum_i \|v_i\|_1\ge 3(1-\delta)=\sqrt3-3\eps.
\]
Since $\|v_i\|_1=w(P_i)\|u_i\|_1$ and $w(\mathcal C)<1$, it follows that
\begin{equation*}\label{eq:3d-deficit}
\sum_i w(P_i)\bigl(\sqrt3-\|u_i\|_1\bigr)<3\eps.
\end{equation*}

For a parameter $\alpha>1$ to be fixed at the end of the proof, we partition the planks in $\mathcal{C}$ into diagonal and non-diagonal classes: 
   \begin{align*}
        \mathcal{D} &= \{P_i:\|u_i\|_1\ge \sqrt 3-3\alpha\eps\}\\
        \mathcal{N} &= \{P_i:\|u_i\|_1< \sqrt 3-3\alpha\eps\}.
    \end{align*}
Using an identical argument to \Cref{claim: wN}, we  derive 
\begin{equation}\label{eq:wN3d}
w(\mathcal N) \le \frac 1 \alpha.
\end{equation}
Next, take a thickness parameter $a\in(0,\frac12)$. Provided $3\alpha\eps<\frac23(\sqrt2-1)$,
Lemma~\ref{lemma:uncovered-volume} gives
\[
\vol(A_a\cap\mathcal D)
\le
3aD(3\alpha\eps)\,w(\mathcal D),
\]
while Lemma~\ref{lem:uncovered-volume-general} yields
\[
\vol(A_a\cap\mathcal N)
\le
\sqrt2\,w(\mathcal N).
\]
If, moreover, $a\le\min\{\delta,1-\delta\}$, then
\begin{equation*}\label{eq:hole-volume-3d}
\vol(A_a\cap\delta Q)
\le
\delta^3-(\delta-a)^3.
\end{equation*}
Indeed, writing $\delta Q=I_1\times I_2\times I_3$, each interval
$I_j\subset[0,1]$ has length $\delta$ and satisfies
\[
|I_j\cap(a,1-a)|\ge\delta-a,
\]
thus the hole $(a,1-a)^3$ covers at least $(\delta-a)^3$ volume of $\delta Q$.

Analogously to the planar case, we will reach a contradiction by choosing the parameters so that the volume covered by members of $\C$, and $\delta Q$ within $A_a$ is strictly smaller than $\vol(A_a)$.

For ease of notation, set $\mu:=3aD(3\alpha\eps).$ When choosing the parameters we will ensure $\mu<\sqrt2$ holds. Combining the volume bounds above with $w(\mathcal D)<1-w(\mathcal N)$ and \eqref{eq:wN3d} yields
\begin{align*}
\vol\bigl(A_a\cap(\mathcal C\cup\delta Q)\bigr) & \leq \vol(A_a\cap \D) +\vol(A_a\cap \N)+\vol(A_a\cap \delta Q) \\
&\le
\mu\,w(\mathcal D)+\sqrt2\,w(\mathcal N)
+\delta^3-(\delta-a)^3\\
&<
\mu+(\sqrt2-\mu)w(\mathcal N)
+\delta^3-(\delta-a)^3\\
&\le
\mu\left(1-\frac1\alpha\right)
+\frac{\sqrt2}{\alpha}
+\delta^3-(\delta-a)^3.
\end{align*}
Consequently, since $\vol(A_a)=1-(1-2a)^3$, a contradiction follows if
\begin{equation}\label{eq:U3}
V(\eps,a,\alpha)
:=1-(1-2a)^3
 -3aD(3\alpha\eps)\left(1-\frac1\alpha\right)
-\frac{\sqrt2}{\alpha}-\delta^3+(\delta-a)^3 >0.
\end{equation}

The function $V$ is continuous for $0<a<\frac{\sqrt2}{3\sqrt3}$,
$\alpha>1$ and $0\le3\alpha\eps<\frac23(\sqrt2-1)$. By containment, proving the theorem for $\varepsilon=\varepsilon_3$ also proves it for every $0<\varepsilon\le\varepsilon_3.$ Therefore, it remains to prove that for some $\eps > 0$ we may choose $a>0$ and $\alpha >1$ subject to the constraints
\[
a<\frac{\sqrt2}{3\sqrt3},
\qquad
3\alpha\eps<\frac23(\sqrt2-1),
\qquad
3aD(3\alpha\eps)<\sqrt2
\]
so that $V(\eps, a, \alpha) > 0$. For any $0< a < \frac{\sqrt2}{3\sqrt3}$, by \eqref{eq:betazero},
\[
V(0,a,\alpha)=1-(1-2a)^3- 3\sqrt3\,a -\frac 1 \alpha (\sqrt{2} - 3 \sqrt{3} a) -\delta_0^3+(\delta_0-a)^3.
\]
Therefore, 
\begin{align*}
\lim_{\alpha \to \infty}V(0,a,\alpha) & =
1-(1-2a)^3- 3\sqrt3\,a -\delta_0^3+(\delta_0-a)^3\\
&=(2-\sqrt3)a-(9+\sqrt3)a^2+7a^3.
\end{align*}
Since $2-\sqrt3>0$, this expression is positive for every sufficiently
small $a>0$. For any such $a$ in the prescribed domain,  choose $\alpha>1$ sufficiently large that
\[
V(0,a,\alpha)>0.
\]
Finally, by continuity of $V$ and $D$, choose $\eps_3>0$
sufficiently small so that 
\[
V(\eps_3,a,\alpha)>0,
\qquad
3\alpha\eps_3<\frac23(\sqrt2-1),
\qquad
3aD(3\alpha\eps_3)<\sqrt2.
\]
For example, let 
\[\ee_3 = 8\times 10^{-8},\quad \alpha = 2850, \quad a=0.0068.\]
Then \eqref{eq:U3} shows that a positive amount of $A_a$ is uncovered
by $\mathcal C\cup\delta Q$, contradicting the covering assumption. This finishes the proof of Theorem~\ref{thm:main} for $d=3$.
\end{proof}

In order to provide a quantitative value for $\eps_3$ above, we numerically find triples $(\eps, a, \alpha)$ with $V(\eps, a, \alpha)>0$. The results are shown in \Cref{fig:heatmap3}. The numerically estimated threshold of $\ee$ occurs at the plus sign in \Cref{fig:heatmap3}, with coordinates
\begin{align*}
    \ee_{\text{max}} &= 8.616\times 10^{-8},\quad \al = 2848.677,\quad a = 0.006832.
\end{align*}

\begin{figure}[h!]
    \centering
    \includegraphics[scale=0.5]{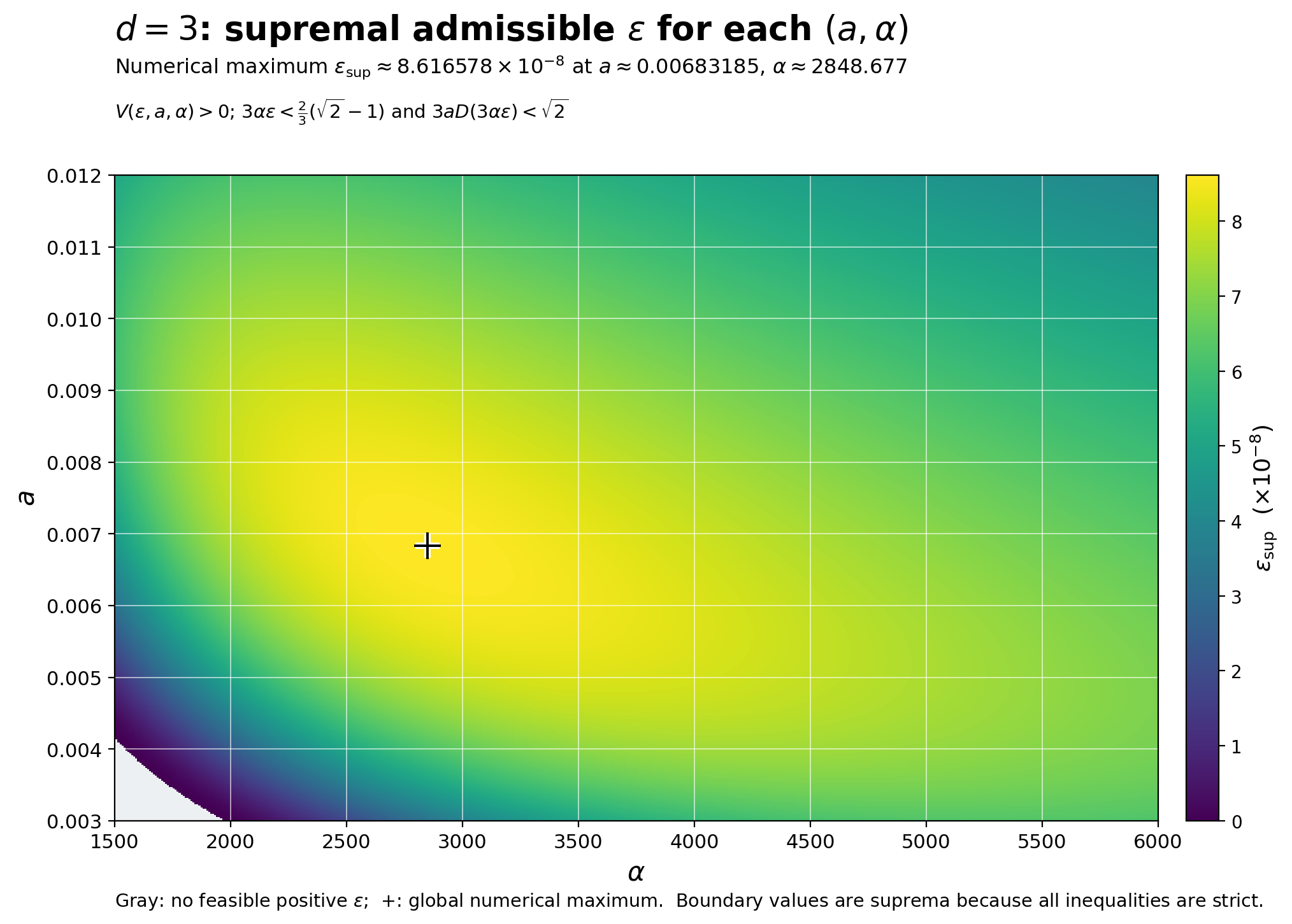}
    \caption{Heatmap of maximal $\ee$ such that $V(\eps, a, \alpha)>0$ for given pairs $(a, \al)$.}
    \label{fig:heatmap3}
\end{figure}

\section{Conclusion}

We note that the techniques we use to improve upon the existing bounds of Bezdek and Smurov, Bogataya, and Bogatyi in dimensions 2 and 3 are not sharp. Thus, to achieve the bound of $1-\frac1d$ in Conjecture~\ref{conj:bezdek}, stronger estimates or additional ideas would be needed.

\medskip

The present slicing estimates do not extend to $d\ge4$. Intersecting a diagonal plank with a two-dimensional coordinate plane increases its width by $\sqrt{d/2}$. Combining the planar collar estimates consequently gives
\[
\vol_d(A_a\cap\mathcal D)\le d\sqrt d\,a\,w(\mathcal D)
\]
for perfectly diagonal planks: there are $d \choose 2$ such tubes, each of which contributes at most $2 a \sqrt{d} w(\mathcal D)$ covered volume, while every collar point belongs to at least $d-1$ such tubes, which yields the above estimate. However, for $d\ge4$,
\[
\vol_d(A_a)=1-(1-2a)^d<2da\le d\sqrt d\,a.
\]
Thus, the bound is already too large when the total width approaches $1$, even before accounting for the inner cube.

Extending our stability argument to dimensions $d\ge4$ would require sharper volume estimates for boundary collars covered by nearly diagonal planks. For the centered cube annulus $A_a\subset\R^d$ defined by \eqref{eq:Adef} and a plank $P$ whose unit normal $u$ satisfies $\|u\|_1\ge\sqrt d-\beta$, this would entail establishing a sufficiently sharp bound of the form
\[
\vol_d(A_a\cap P)\le F_d(a,\beta)\,w(P),
\]
uniformly over all such planks $P$.

\section*{Acknowledgements}
This work was conducted as part of Undergraduate Research at the Budapest Semesters in Mathematics Program during the fall of 2025. We thank BSM for its support and for providing a good working environment. AI tools were used solely for text editing and the creation of \Cref{fig:heatmap3}. %Codex was used to generate the heatmaps.

%\nocite{*}
\bibliographystyle{plain}
\bibliography{PlankCube_refs.bib}

\bigskip

\footnotesize{Research of G. A. was partially supported by the ERC Advanced Grant "GeoScape" no.  882971, by Hungarian National Research (NKFIH) grants no. KKP-133819, 147145, 147544, and 150151, which has been implemented with the support provided by the Ministry of Culture and Innovation of Hungary from the National Research, Development and Innovation Fund, financed under the ADVANCED-24 funding scheme. This research was funded by the grant 2024-1.2.8-TÉT-IPARI-CN-2025-00011,
with the support provided by the National Research,
Development and Innovation Office from the National Research,
Development and Innovation Fund, and financed under the
2024-1.2.8-TÉT-IPARI-CN funding scheme. }

\bigskip

\noindent
{\sc Gergely Ambrus}
\smallskip

\noindent
{\em Bolyai Institute, University of Szeged, Hungary,\\ and\\ 
 Alfréd Rényi Institute of Mathematics, Budapest, Hungary
 }
\smallskip

\noindent
e-mail address: \texttt{ambrus@server.math.u-szeged.hu; ambrus@renyi.hu}

\bigskip

\noindent
{\sc Matthew Corbo}
\smallskip

\noindent
{\em University of California, Berkeley, Berkeley, CA, United States
 }
\smallskip

\noindent
e-mail address: \texttt{corbo@berkeley.edu}

\bigskip

\noindent
{\sc Dakota Nobuyuki Goto}
\smallskip

\noindent
{\em Fordham University, New York, NY, United States
 }
\smallskip

\noindent
e-mail address: \texttt{dgoto@fordham.edu}

\bigskip

\noindent
{\sc Desmond Leitz}
\smallskip

\noindent
{\em Stony Brook University, Stony Brook, NY, United States
 }
\smallskip

\noindent
e-mail address: \texttt{desmond.leitz@stonybrook.edu}

\bigskip

\noindent
{\sc Kai Mawhinney}
\smallskip

\noindent
{\em Harvey Mudd College, Claremont, CA, United States
 }
\smallskip

\noindent
e-mail address: \texttt{kmawhinney@hmc.edu}

\end{document}